\documentclass[pdf,generic]{imsart}
\usepackage{enumerate}
\usepackage{amsthm}
\usepackage{amsmath}
\usepackage{bm}

\RequirePackage[T1]{fontenc}
\RequirePackage{amsthm,amsmath}
\usepackage{amsfonts}
\RequirePackage[numbers]{natbib}
\RequirePackage[colorlinks,citecolor=blue,urlcolor=blue]{hyperref}
\usepackage[utf8]{inputenc}
\pubyear{0000}
\volume{0}
\issue{0}
\firstpage{0}
\lastpage{0}

\startlocaldefs
\numberwithin{equation}{section}
\theoremstyle{plain}
\newtheorem{thm}{Theorem}[section]

\newtheorem{prop}{Proposition}[section]

\endlocaldefs

\DeclareMathOperator*{\argmin}{argmin}

\DeclareMathOperator*{\rank}{rank}

\def\Det#1{\left|#1\right|}

\begin{document}

\begin{frontmatter}
% \title{A Sample Document\thanksref{T1}}
\title{A note on the existence of the maximum likelihood estimate
in variance components models}

%On computing maximum likelihood 
%estimates in mixed linear normal models with two variance components (draft version)}
\runtitle{Variance components models}
%\thankstext{T1}{Footnote to the title with the `thankstext' command.}

\begin{aug}
\author{\fnms{Mariusz} \snm{{Grządziel}\ead[label=e1]{mariusz.grzadziel@up.wroc.pl}}}
\and
\author{\fnms{Andrzej} \snm{Michalski}\ead[label=e2]{andrzej.michalski@up.wroc.pl}}
\address{
Department of Mathematics,\\
Wroc\l aw University of Environmental and Life Sciences\\
Grunwaldzka 53, 50 357 Wroc\l aw, Poland
\printead{e1,e2}}
%\author{\fnms{Third} \snm{Author}
%\ead[label=e3]{third@somewhere.com}
%\ead[label=u1,url]{www.foo.com}}
%
%\address{Address of the Third author\\
%usually  few lines long\\
%usually few lines long\\
%\printead{e3}\\
%\printead{u1}}

%\thankstext{t1}{Some comment}
%\thankstext{t2}{First supporter of the project}
%\thankstext{t3}{Second supporter of the project}
\runauthor{M.~Grządziel and A.~Michalski}

\affiliation{Wroc\l aw University of Environmental and Life Sciences}

\end{aug}

\begin{abstract}
In the paper, the problem of the existence of the maximum likelihood estimate 
and the REML estimate
in the variance components model is considered.
Errors in the proof of Theorem~3.1 in 
the article of Demidenko and Massam (Sankhy\=a 61, 1999), giving a necessary 
and sufficient condition 
for the existence of the maximum likelihood estimate in this model, are pointed 
out and corrected. A new proof of  Theorem~3.4 in the Demidenko and Massam's 
article, concerning the existence of the REML estimate of variance components, 
is presented.
\end{abstract}


\tableofcontents
\end{frontmatter}

\section{Introduction} 
The paper of Demidenko and Massam \cite{DM99} gives a definitive answer to the problem of the existence of the maximum likelihood estimate in the variance components model. 
Checking the necessary and sufficient condition for its existence 
given in this paper
should be done before using numerical procedures for computing it.
Such a check should be also performed (for stronger reasons!) if we are going to use the algebraic methods for computing this estimate: In this approach all critical points of the log-likelihood function are computed via solving a system of polynomial equations, see the paper of Gross \textit{et. al.} \cite{GDP12}.

The main purpose of this paper is to point out errors in the proof of Theorem~3.1 in \cite{DM99}, giving a necessary and sufficient condition
for the existence of the maximum likelihood estimate in the variance components model, and to correct them. 
We also show that 
the result of
this theorem
can be extended to the case when  the design matrix  is rank deficient. In consequence, we obtain a new proof of Theorem~3.4 in \cite{DM99} which gives a necessary and sufficient condition for the existence of the restricted maximum likelihood (REML) estimate in the variance components model.

The paper is organized as follows.  
In Section \ref{S:Existence},
after recalling 
the Demidenko and Massam's theorem
\cite[Theorem~3.1]{DM99},
we point out two errors in its proof
given by the authors and we correct these errors. The problem of the existence of the REML estimate is 
considered  in Section \ref{S:REML}. 

\subsection{Notation}
For vector $ \bm{x} \in \mathbb{R}^{n} $ we will denote its Euclidean norm by $ \left\| \bm{x} \right\|  $.
For a given $m \times n$ matrix $\bm{A}$, we will denote
by $\bm{A}^{\prime}$ its transpose,
by $\rank(\bm{A})$ its rank
and  by $\mathcal{M}(\bm{A})$ the space spanned by the columns of $ \bm{A} $.
For the given matrices $\bm{A}_{1},\bm{A}_{2},\ldots,\bm{A}_{p}$
of dimension $m \times n_{1}, m \times n_{2}, \dots, m \times n_{p}$, respectively, we will denote by $ [\bm{A}_{1},\bm{A}_{2},\dots,\bm{A}_{p}] $ the partitioned  $m \times (n_{1}+n_{2}+\ldots+n_{p})$  matrix consisting of  $ \bm{A}_{1}, \bm{A}_{2},\ldots, \bm{A}_{p} $.
For brevity we will write 
$\mathcal{M}(\bm{A}_{1},\bm{A}_{2},\dots,\bm{A}_{p})$ instead of 
$\mathcal{M}([\bm{A}_{1},\bm{A}_{2},\dots,\bm{A}_{p}])$.
We will write
$|\bm{B}|$ for the determinant of a square matrix $\bm{B}$,
$\bm{I}_{n}$ for the 
identity matrix of order $n$.
The matrix of order $ m \times n $ in which every entry is equal to $ 0 $ we will denote by $ \bm{0}_{m \times n} $. If the order of the matrix is clear from the context we will denote the zero matrix simply by $ \bm{0} $.
The $ n $-dimensional vector having all coordinates equal to $ 1 $ we will  denote by $ \bm{1}^{(n)} $,
the  $ n $-dimensional vector having all coordinates equal to $ 0 $ we will denote by $ \bm{0}^{(n)} $.

For an $ m \times n  $ matrix $\bm{G} $ 
and a subspace $ \mathcal{V} $
of $ \mathbb{R}^{n} $, we will denote by $\bm{G}\mathcal{V} $ the image of the transformation corresponding to $ \bm{G} $ of the subspace $ \mathcal{V} $.  We will use the notation $\bm{Y} \sim  \mathcal{N}(\bm{\mu}, \bm{\Sigma})$ if the random vector $ \bm{Y} $ has the multivariate normal distribution with the mean vector $ \bm{\mu} $ and
the variance-covariance matrix $\bm{ \Sigma} $.
For a real-valued function $f$ with domain $S$
we define
\[ \argmin_{x \in S}f(x):=\{z \in S: f(z) \le f(x) \, \text{for all} \, x \in S \}   . \]

\section{The existence of the maximum likelihood estimate} \label{S:Existence}
The variance components model
considered in the Demidenko and Massam's paper \cite{DM99} can be expressed in the form
\begin{equation} \label{E:ModelVarComp}
\bm{Y}=\bm{X}\bm{\beta}+\sum_{i=1}^{r}\bm{Z}_{i} \bm{u}_{i}+\bm{\epsilon},
\end{equation}
where 
$ \bm{Y} $ is an $n \times 1$ random vector,
$\bm{X}$ is an $n \times m$ known design matrix of full column rank, $m<n$,
$\bm{\beta}$ is an $m \times 1$ vector of fixed parameters,
$ \bm{Z}_{i} $  are known  
$ n \times k_{i} $ matrices
such that 
$ 
\sum_{i=1}^{r} k_{i}<n
$
and
$ \bm{u}_{i} \sim \mathcal{N}(\bm{0},\sigma_{i}^{2}\bm{I}_{k_{i}})$, 
$ i=1,\ldots,r $,
while 
$ \bm{\epsilon} \sim \mathcal{N}(\bm{0},\sigma_{0}^{2} \bm{I}_{n})$. 
We assume that all random terms are stochastically independent. The parameter space is equal to
\begin{equation}
\bm{\Theta} :=   \{ \bm{\theta}:=(\bm{\beta}, \sigma_{0}^{2}, \sigma_{1}^{2},\ldots,\sigma_{r}^{2})=(\bm{\beta},\bm{\sigma}^{2}) \in \mathbb{R}^{m} \times (0,\infty) \times [0,\infty)^{r} \},
\end{equation}
where 
$ \bm{\sigma}^{2} := (\sigma_{0}^{2},\sigma_{1}^{2},\ldots,\sigma_{r}^{2}) $ is the vector of variance components.
It can be seen that
the covariance matrix of the vector $ \bm{Y} $ 
can be expressed as
\[ 
\bm{V}(\bm{\sigma}^{2}):=\bm{ \sigma}_{0}^{2}\bm{I}_{n}+\sum\limits_{i=1}^{r}\sigma_{i}^{2}\bm{Z}_{i}\bm{Z}_{i}^{\prime}.
\]
Minus twice the log-likelihood function is given, up to an additive constant, by
\begin{equation} \label{E:Loglikelihood}
l(\bm{\beta},\bm{\sigma}^{2},\bm{Y}):=\log\Det{\bm{V}(\bm{\sigma}^{2})}+(\bm{Y}-\bm{X}\bm{\beta})^{\prime}\bm{V}^{-1}(\bm{\sigma}^{2})(\bm{Y}-\bm{X}\bm{\beta}). 
\end{equation}
For a given realization $ \bm{y} $ of the observation vector $ \bm{Y} $ the 
maximum likelihood estimate of  $\bm{\theta} \in \bm{\Theta}  $ is defined as
\begin{equation} \label{E:MLEstimate0}
    \argmin_{\bm{\tau} \in \bm{ \Theta}} {l(\bm{\tau},\bm{y})}.
\end{equation}
Put 
\begin{align}
\kappa_{0} &:= \sigma_{0}^{2}, \quad  \kappa_{i}  := \frac{\sigma_{i}^{2}}{ \sigma_{0}^{2}}, \quad 1=1,2,\dots,r, \label{E:param1}    \\
\bm{\kappa}  &:= (\kappa_{1},\ldots,\kappa_{r}), \quad
\bm{\tilde V}( \bm{\kappa})  := \bm{I}_{n}+ \sum\limits_{i=1}^{r} \kappa_{i} \bm{Z}_{i}\bm{Z}_{i}^{\prime}. \label{E:param2}
\end{align}
Since $ \bm{V}(\bm{\sigma}^{2})=\kappa_{0} \bm{\tilde V}(\bm{\kappa}) $,
we will refer to $  \bm{\tilde V}(\bm{\kappa}) $ as the scaled covariance matrix of the vector $ \bm{Y} $. 
The parameter space 
and the counterpart of the function $ l $
corresponding to this new parametrization are given by
\begin{align}
\bm{\tilde \Theta} & := 
\{ \bm{ \tilde \theta } := (\bm{\beta},\kappa_{0}, \kappa_{1},\ldots,\kappa_{r}) \in \mathbb{R}^{m} \times (0,\infty) \times [0,\infty)^{r} \},\\
\label{E:LoglikelihoodAlt}
\tilde l(\bm{\beta},\bm{\kappa},\bm{Y})&:=n \log \kappa_{0}+\log|\bm{\tilde V(  \bm{\kappa})}|+\kappa_{0}^{-1}(\bm{Y}-\bm{X}\bm{\beta})^{\prime}\bm{\tilde V}^{-1}(\bm{\kappa})(\bm{Y}-\bm{X}\bm{\beta}).
\end{align}
The function $\tilde l$ is equal, up to an additive constant, to minus twice the log-likelihood function 
expressed in terms of $\bm{\beta},\bm{\kappa}$ and $\bm{Y}  $.
%\section{Existence of the maximum likelihood estimate} \label{S:Existence}
If $ \bm{y} $ is a given  realization of the observation vector $ \bm{Y} $, then
the maximum likelihood estimate of $\bm{ \tilde \theta} \in \bm {\tilde \Theta}$
can be expressed as 
\begin{equation} \label{E:MLEstimate}
    \argmin_{\bm{\tau} \in \bm{\tilde \Theta}} {\tilde l(\bm{\tau},\bm{y})}.
\end{equation}
%where  $ \bm{y} $ is a given  realization of the observation vector $ \bm{Y} $.
We will say that the maximum likelihood estimate of the parameters
exists for the given realization $ \bm{y} $ of the observation vector $ \bm{Y} $ 
if and 
only if the set
(\ref{E:MLEstimate})
(or the set (\ref{E:MLEstimate0})) is not empty.

A definitive answer to the existence problem in the variance components model is given  by Demidenko and Massam in  \cite[Theorem 3.1]{DM99}. 
Before stating it, let us define the following quantities:
\begin{align}
\bm{Z}:=[\bm{Z}_{1},\bm{Z}_{2},\dots,\bm{Z}_{r}], \quad
M&:=\mathcal{M}(\bm{X}), \quad H:=\mathcal{M}(\bm{Z}),
\label{E:ZMH} \\ s_{X,Z}&:=\bm{y}^{\prime}\bm{P}_{(H+M)^{\perp}}\bm{y},
\end{align}
where $ \bm{P}_{(H+M)^{\perp}} $ stands for the orthogonal projection from $ 
\mathbb{R}^{n} $  to the orthogonal subspace to $H+M$.

%Now we are ready to state the above mentioned theorem \cite[Theorem 3.1]{DM99}.
\begin{thm}[{\cite[Theorem~3.1]{DM99}}] \label{T:ML}
The maximum likelihood estimate in the variance components model (\ref{E:ModelVarComp}) exists if and only if
\begin{equation} \label{C:Existence}
\bm{y} \notin \mathcal{M}(\bm{X},\bm{Z}).
\end{equation}

\end{thm}

The proof of this theorem can be sketched as follows:
Theorem \ref{C:Existence} is an immediate consequence of the following facts mentioned in \cite[p.~436--437]{DM99} presented below in the form of
\begin{prop} \label{P:facts}
If for a fixed $ \bm{y} \in \mathbb{R}^{n}$ the condition (\ref{C:Existence}) is satisfied, then:
\begin{enumerate}[(a)]
\item 
$ \tilde l( \bm{\tilde \theta},\bm{y}) \ge n \log s_{X,Z}-n \log n +n$.
\item 
For a given $ A> n \log s_{X,Z}-n \log n +n$  there exists a compact set $ C_{A} \subset \bm{\Theta} $ such that 
\begin{equation}  \label{I:condition}
\tilde l(\bm{\tilde \theta},\bm{y}) \ge A \quad \text{for each} \quad \tilde \theta \in \bm{\tilde \Theta} \setminus C_{A}.
\end{equation}
% see \cite[p.~436]{DM99}
\end{enumerate}
Otherwise:
\begin{enumerate}[(a)]
\setcounter{enumi}{2}
\item 
The infimum of $ \tilde l(\bm{\tilde \theta},\bm{y})$ in $ \bm{\tilde \Theta} $ is $ - \infty $.
\end{enumerate}

\end{prop}

The proofs of the above facts can be found in the course of the proof of Theorem 3.1 in \cite{DM99}. In
the next subsection
we will discuss their correctness.

\subsection{The errors in the proof and their corrections}

\subsubsection{Representation of the scaled covariance matrix}

The parts of the proof of
Theorem 3.1 in \cite{DM99} that correspond to the parts (a) and (c) of Proposition \ref{P:facts} rely on a certain algebraic fact concerning the representation of the scaled covariance matrix
$\bm{\tilde V}(\bm{\kappa})$
of the observation vector $ \bm{Y} $.
Unfortunately, the proof of this fact, stated as
Proposition 3.2 in {\cite{DM99}, is not quite correct.  

\paragraph[Comments on Demidenko and Massam's proof]{Comments on Demidenko and Massam's proof}

We will restate Proposition 3.2 in \cite{DM99} as the following

\begin{prop}[Demidenko and Massam] \label{P:Vkappa}
Let us denote the dimension of the space $ H $ defined in (\ref{E:ZMH}) by $ q $.
There exists an $ n \times n $ orthogonal matrix $ \bm{U} $ and $ q \times q $   matrices $ \bm{A}_{1},\ldots,\bm{A}_{r} $
satisfying the conditions
\begin{equation} \label{E:Amatrices}
\bm{A}_{i} \, \text{is positive definite}, \, i=1,\ldots,r,
\end{equation}
such that
\begin{equation} \label{E:Vkappa}
\bm{\tilde V}(\bm{\kappa})=\bm{U}
\begin{bmatrix}
\bm{I}_{q}+\sum_{i=1}^{r}\kappa_{i}\bm{A}_{i} &\bm{0} \\
\bm{0} &\bm{I}_{n-q}
\end{bmatrix}\bm{U}^{\prime}.
\end{equation}
Moreover, if $ c=\kappa_{1}=\kappa_{2}=\ldots=\kappa_{r} $, there exists an $ n 
\times n $ orthogonal matrix $ \bm{U} $ and a $ q \times q $ diagonal matrix $ 
\bm{D} $ with positive diagonal elements satisfying
\begin{equation} \label{E:Vkappa1}
\bm{\tilde V}(\bm{\kappa})=\bm{U}
\begin{bmatrix}
\bm{I}_{q}+c\bm{D}  & \bm{0}\\
\bm{0} & \bm{I}_{n-q}
\end{bmatrix}\bm{U}^{\prime}.
\end{equation}
\end{prop}

The proof of
this proposition 
given in \cite[p.~442]{DM99}
is not correct:
The $ q \times q $ matrices $\bm{A}_{1},\ldots,\bm{A}_{r}  $ constructed during the proof satisfy the condition
$\rank(\bm{A}_{i})= \rank(\bm{Z}_{i})$, $i=1,2,\ldots,r $, 
and if there exists $ i $ such that $ \rank(\bm{Z}_{i})<q $
(the ranks of the matrices $ \bm{Z}_{i} $ are not equal in 
the general case)
the matrix $\bm{A}_{i}  $ will have rank less than $ q $ which will contradict 
its positive definiteness.

\paragraph{The correction}
Let us consider the weakened version of Proposition \ref{P:Vkappa} in which the condition (\ref{E:Amatrices}) is replaced by the following condition 
\[
\bm{A}_{i} \, \text{is non-negative definite}, \, i=1,\ldots,r, \, \text{and} \,    \sum_{i=1}^{r}\bm{A}_{r}
\, \text{is positive definite.}
\]
It can be verified that using  the arguments from the proof  of Proposition 3.2 in \cite[p.~442]{DM99} we can obtain 
the mentioned above weakened version of
Proposition \ref{P:Vkappa}
which is sufficient for the
purposes of the proof of \cite[Theorem 3.1]{DM99}.

 \subsubsection{Attaining the supremum of the log-likelihood function}

 \paragraph[Comments on Demidenko and Massam's proof]{Comments on Demidenko and 
 Massam's proof}
 Let us now discuss the part of the proof of Theorem 3.1 in \cite{DM99} that corresponds to the part (b) of Proposition \ref{P:facts}
 and concerns attaining the infimum of the function $ \tilde l $ --- or attaining the supremum of the likelihood function.
It contains the following statement: 
It is possible to
construct a set $ C_{A} $ satisfying the condition
(\ref{I:condition})
 by:
\begin{itemize}
\item choosing $ \epsilon \in (0,1)$ such that 
\begin{align}
\text{for} \quad  \kappa_{0} \in (0,\epsilon), \quad & n \log \kappa_{0}+\kappa_{0}^{-1}s_{X,Z} \ge A; \label{C:cond1} \\
\text{for} \quad \kappa_{0} > \epsilon^{-1}, \quad & n \log \kappa_{0} \ge A; \label{C:cond2}  \\
\text{for} \quad \epsilon \le \kappa_{0} \le  \epsilon^{-1}, \quad & s_{X,Z}
 \ge (A-n \log \epsilon) \epsilon^{-1}; \label{C:cond3}
\end{align}

\item 
choosing any bounded set $ B \subset M $;
\item 
choosing $ b>0 $ such that 
\begin{equation} \label{C:condB}
\text{for} \quad \max_{1 \le i \le r} \kappa_{i} \ge b, \quad \log |\tilde V(\bm {\kappa})| \ge A - n \log \epsilon;
\end{equation}
\item 
putting 
$
C_{A} := \{\bm{\beta}:\bm{X}\bm{\beta} \in B \} \times(\epsilon,\epsilon^{-1}) \times [0,b)^{r},
$
\end{itemize}
see Demidenko and Massam \cite[p.~436]{DM99}.
 
Unfortunately, this claim is not true.
To see this, let us consider a set $ C^{0} $ constructed by:
%For $ y $ satisfying the condition (\ref{C:Existence})
\begin{itemize}
\item 
choosing $A $ such that
$ A>m+\epsilon_{0} $,
where $ m $ is 
the infimum of $ \tilde l $ in $\bm{ \tilde \Theta }$
while $\epsilon_{0}  $ is a given positive number;

\item 
choosing $ \bm{ \tilde \theta}_{0} = 
(\bm{\beta}_{0},\kappa_{0}^{0},\kappa_{1}^{0},\ldots,\kappa_{r}^{0}) $ such 
that $ \tilde l(\bm{\tilde \theta_{0}}) \le m+\epsilon_{0}  $;
\item 
choosing a bounded set $ B_{0} \subset M$ such that 
$ \bm{X} \bm{\beta}_{0} \notin B_{0} $ ;
\item 
%for $ A>m+\epsilon_{0} $
choosing $ \epsilon_{1} \in (0,1) $ and $ b_{1} >0  $ such that for $ 
\epsilon=\epsilon_{1} $,
$ b=b_{1} $ and $ B=B_{0} $ the conditions 
(\ref{C:cond1})--(\ref{C:condB}) are satisfied;
\item putting
\[
C^{0}:=\{\bm{\beta}:\bm{X}\bm{\beta} \in B_{0} \} \times(\epsilon_{1},\epsilon_{1}^{-1}) \times [0,b_{1})^{r}.
\]
\end{itemize}
According to Demidenko and Massam \cite[p.~436]{DM99} the condition (\ref{I:condition}) is satisfied for $ C_{A}=C^{0} $ (for the chosen $ A $).
This implies that  $ \tilde l(\bm{ \tilde \theta_{0}}) \ge A > m+\epsilon_{0}$, and we have obtained a contradiction.

\paragraph{The correction}

We will now present the corrected version of the proof of the part (b) of Proposition \ref{P:facts}. \\
\begin{proof}[Proof of the part (b) of Proposition \ref{P:facts}] 
Let us choose 
\begin{enumerate}[(i)]
\item 
$ \epsilon \in (0,1) $ such that the conditions (\ref{C:cond1}) and (\ref{C:cond2}) are satisfied; 
\item
$ b $ such that the condition (\ref{C:condB}) is satisfied;
\item $ t $ such that
for $ \bm{\beta} \in \mathbb{R}^{m} $ satisfying $\left\| \bm{\beta} \right\| \ge t   $:
\begin{align}
\text{If} \, &  \kappa_{0} \in [\epsilon,\epsilon^{-1}]
\, \text{and} \, 
\max_{1 \le i \le r} \kappa_{i} < b, \, \text{then}  \label{E:kappa}  \\ 
&\kappa_{0}^{-1}(\bm{y}-\bm{X}\bm{\beta})^{\prime}\bm{\tilde V}^{-1}(\bm{\kappa})(\bm{y}-\bm{X}\bm{\beta}) \ge A - n \log \epsilon.
\end{align}

\end{enumerate} 
The existence of $ b $ mentioned in (ii) stems from the fact that 
\[ 
\lim\limits_{\lambda \rightarrow \infty} \log | \bm{\tilde V}(\bm{\kappa}+ \lambda \bm{v})|=\infty,
\]
where $ \bm{v} \in [0,\infty)^{r} \setminus \{ \bm{0}^{(r)} \}   $.
% here $\bm{z}_{0}^{r}$ denotes 
% $ r $-dimensional vector with all coordinates  equal to $ 0 $. 
To prove the existence of $ t $ in (iii) we can use the fact that
if the condition (\ref{E:kappa}) is satisfied, then
the matrix
$\bm{V}^{-1}(\bm{\kappa})-\bm{V}^{-1}(b \bm{1}^{(r)} )$
is non-negative definite, see \cite[p.~70]{Rao73}.
% \cite[p.~451]{RR98} or \cite[p.~24]{Jar00}. 
It can be seen that $ C_{A}:=[-t,t]^{m} \times [\epsilon,\epsilon^{-1}] \times [0,b]^{r} $ satisfies the condition (\ref{I:condition}).
\end{proof}
From the course of the above proof follows immediately

\begin{prop} \label{P:ML-infty}
Let us assume that the sequence $(\bm{\tilde \theta}_{n})$ of elements of 
the parameter space $ \bm{\tilde \Theta} $,
$ \bm{\tilde \theta}_{n}=(\bm{\beta}^{(n)},\kappa_{0}^{(n)},\bm{\kappa}^{(n)})
$,
satisfies at least one of the following conditions : $\left\| \bm{ \beta}^{(n)} \right\| \rightarrow \infty$, 
$\kappa_{0}^{(n)}\rightarrow0$, $\kappa_{0}^{(n)}\rightarrow\infty$ or $\left\| 
\bm{\kappa}^{(n)} \right\| \rightarrow \infty$. If $ \bm{y} $ satisfies the 
condition (\ref{C:Existence}), then
\[ 
\lim\limits_{n \rightarrow \infty}
\tilde l( \bm{\tilde \theta}_{n}, \bm{y})=\infty. 
\]
\end{prop}

\subsection{The case of rank deficient design matrix } \label{S:NonFull}

Throughout the paper, except this subsection, we assume that the design matrix $ \bm{X} $ has a full column rank. It can be shown that
\begin{prop} \label{P:NonFull}
Theorem \ref{T:ML} remains valid 
if the assumption "$\bm{X}$ is a matrix of full column rank" is dropped.
\end{prop}

\begin{proof}
Let us consider the case $ \bm{X} \neq \bm{0}_{n \times m} $ first. Choose a 
matrix  of full column rank 
$ \bm{X}_{1} $
such that $\mathcal{M}(\bm{X}_{1}) =
\mathcal{M}(\bm{X})$. Observe that for a given realization $ \bm{y} $ of the observation vector $ \bm{Y} $ the maximum likelihood estimate of $ \bm{\sigma^{2}} $
in the new model (with the design matrix $\bm{X}_{1}$) exists if and only if it exists in the model with  the design matrix $ \bm{X} $.
It follows from Theorem
\ref{C:Existence} that the maximum likelihood estimate exists in the former 
model if and only if
$\bm{y} \notin \mathcal{M}(\bm{X}_{1},\bm{Z})=\mathcal{M}(\bm{X},\bm{Z})$,
which completes the proof of the part of the proposition
concerning the case of non-zero design matrix.

Let us now turn our attention to the case $ \bm{X}=\bm{0}_{n \times m} $,
when 
we can assume, without loss of generality, that $\bm{\beta}=\bm{\beta}_{0} $, 
where $\bm{\beta}_{0}  $ is an arbitrarily chosen element of $ \mathbb{R}^{m} $.
We can thus replace the sample space $\bm{\tilde \Theta}$ by 
$\bm{\tilde \Theta}_{1}:=
\{  \{ \bm{\beta}_{0 }\} \times (0,\infty) \times [0,\infty)^{r} \}
$. It can be seen that the analogue of Proposition \ref{P:facts}
corresponding to the "modified sample space" $\bm{\tilde \Theta}_{1}$ holds 
true. This completes the proof.
\end{proof}

\section{The existence of the REML estimate} \label{S:REML}

Let $ \bm{K} $ be an $  n \times (n-m) $ matrix of rank $ n-m $ satisfying the condition 
\begin{equation} \label{E:K}
\bm{K}^{\prime}\bm{X}=\bm{0}_{(n-m) \times m};
\end{equation}
let us recall that throughout the paper, except Subsection \ref{S:NonFull}, we 
assume that the design matrix $ \bm{X} $ has a full column rank. The vector $ 
\bm{K}^{\prime} \bm{Y} $ follows the multivariate normal distribution 
$\mathcal{N}(\bf{0},\bm{K}^{\prime}\bm{V}(\bm{\sigma}^{2})\bm{K})$.
The function
\[ 
l_{K}(\bm{\sigma}^{2},\bm{Y}):=\log\left| \bm{K}^{\prime} \bm{V}(\bm{\sigma}^{2}) \bm{K}  \right| + \bm{Y}^{\prime}\bm{K}(\bm{K}^{\prime} \bm{V}(\bm{\sigma}^{2}) \bm{K}   )^{-1}\bm{K}^{\prime}\bm{Y}   
\]
is equal, up to an additive constant, to minus twice the log-likelihood function  based on the vector $ \bm{K}^{\prime} \bm{Y} $
\cite[p.~13]{Jia07}. Let $ \bm{y} $ be a realization of the observation vector~$ \bm{Y} $.
The REML estimate of the vector of variance components $\bm{ \sigma}^{2} $ is defined as
\begin{equation} \label{E:REMLestimate}
    \argmin_{\sigma^{2} \in \Xi} { l_{K}(\bm{\sigma}^{2},\bm{y})},
\end{equation}
where $\bm {\Xi}:=(0,\infty) \times [0,\infty)^{r}  $.
It can be shown that the set (\ref{E:REMLestimate}) does not depend on the choice of the matrix $ \bm{K} $ satisfying the condition (\ref{E:K}), see \cite[p.~48]{Jia07}. It follows immediately from Proposition \ref{P:NonFull}
that the necessary and sufficient condition for the existence of the REML estimate of the vector $ \bm{ \sigma}^{2}  $ has the form
$ 
\bm{K}^{\prime}\bm{y} \notin \bm{K}^{\prime}H
$,
which is equivalent to
$\bm{K}\bm{K}^{\prime}\bm{y} \notin \bm{K}\bm{K}^{\prime}H$. The last condition is in turn equivalent to 
\begin{equation} \label{C:REML}
\bm{y} \notin \bm{N}H, \quad \text{where} \quad \bm{N}=\bm{I}_{n}-\bm{X}(\bm{X}^{\prime}\bm{X})^{-1}\bm{X}^{\prime}.\end{equation}

We have thus proved Theorem 3.4 from \cite{DM99}, which 
we restate as

\begin{thm} \label{T:REML}
The REML estimate of the vector of the variance components~$ \bm{\sigma}^{2} $
in the variance components model (\ref{E:ModelVarComp}) exists if and only if
the condition (\ref{C:REML}) is satisfied.
\end{thm}

%Let us note that the proof of this theorem in \cite{DM99}  was omitted,
%while our proof is more detailed.

It can be also verified that the following analogue of Proposition (\ref{P:ML-infty}) holds true.

\begin{prop}
Let us assume that the sequence $ (\bm{\xi}_{n})$ of elements of the parameter space  $\bm {\Xi}$, 
$ \bm{\xi}_{n}=(\xi_{0}^{(n)},\xi_{1}^{(n)},
\ldots,\xi_{r}^{(n)}) $, satisfies at least one of the following conditions:
\begin{align} 
\lim\limits_{n \rightarrow \infty} \frac{\xi_{k}^{(n)}}{\xi_{0}^{(n)}}=\infty \quad \text{for some} \quad k \in \{1,\ldots,r  \},  \\
\lim\limits_{n \rightarrow \infty} \xi_{0}^{(n)}=0 \quad \text{or} \quad 
\lim\limits_{n \rightarrow \infty} \xi_{0}^{(n)}=\infty.
\end{align}
If $ \bm{y} $ satisfies the condition (\ref{C:REML}), then
\[
\lim\limits_{n \rightarrow \infty} l_{K}(  \bm{\xi}_{n},\bm{y})=\infty.
\]
\end{prop}
It might be expected that this proposition, as well as Proposition \ref{P:ML-infty}, will prove to be useful for future research concerning maximum likelihood estimation in variance components models.

\end{document}